\documentclass[reqno,11pt]{amsart}
\usepackage{amsfonts, amsmath, amssymb, amsthm,   color}
 \usepackage{amsfonts}
\usepackage{amsmath}
\usepackage{bm}
\allowdisplaybreaks[1]

\def\bbm[#1]{\mbox{\boldmath $#1$}}

\newtheorem{thm}{Theorem}[section]
\newtheorem{lemma}[thm]{Lemma}
\newtheorem{prop}[thm]{Proposition}

\newtheorem{ex}[thm]{Example}
\newtheorem{rmk}[thm]{Remark}

\theoremstyle{definition}
\numberwithin{equation}{section}

  \newcommand{\C}{\mathbb{C}}
\newcommand{\cal}{\mathcal }
  
  \newcommand{\N}{\mathbb{N}}
\newcommand{\D}{\Delta}
\newcommand{\rr}{\mathbb{R}}
\newcommand{\intr}{\int_{\R^2}}
\newcommand{\R}{\mathbb{R}}
\newcommand{\al}{\alpha}
\newcommand{\de}{\delta}

\newcommand{\la}{\lambda}
\newcommand{\into}{\int_{B_1}}

\newcommand{\e}{\varepsilon}
\renewcommand{\(}{\left(}
\renewcommand{\)}{\right)}

\newcommand{\beq}{\begin{equation}}
\newcommand{\eeq}{\end{equation}}

\def\bbm[#1]{\mbox{\boldmath $#1$}}

\begin{document}

\title[Non-simple blow-up]{On the construction of non-simple blow-up solutions \\  for the singular Liouville equation with a potential}
\author[Teresa D'Aprile]{ Teresa D'Aprile}
\author[Juncheng Wei]{Juncheng Wei}
\author[Lei Zhang]{Lei Zhang} 
\address[Teresa D'Aprile] {Dipartimento di Matematica, Universit\`a di Roma ``Tor
Vergata", via della Ricerca Scientifica 1, 00133 Roma, Italy.}
\email{daprile@mat.uniroma2.it}
\address[Juncheng Wei]{Department of Mathematics \\ University of British Columbia\\ Vancouver, BC V6T1Z2, Canada} \email{jcwei@math.ubc.ca }
\address[Lei Zhang]{Department of Mathematics\\
        University of Florida\\
        1400 Stadium Rd\\
        Gainesville FL 32611}
\email{leizhang@ufl.edu}

\thanks{ The research of T. D'Aprile is partially supported by the MIUR Excellence Department Project awarded to the Department of Mathematics, University of Rome ``Tor Vergata", CUP E83C18000100006. The research of J. Wei is partially supported by NSERC of Canada. The research of  Lei Zhang is partially supported by a Simons Foundation Collaboration Grant.}

\begin{abstract}
 We are concerned with the existence of blowing-up solutions
 to the following boundary value problem $$-\Delta u= \la V(x) e^u-4\pi N {\bm{\delta}}_0\;\hbox{ in } B_1,\quad	u=0 \;\hbox{ on }\partial B_1,$$  where $B_1$ is the unit ball in $\R^2$ centered at the origin, $V(x)$ is a positive  smooth  potential,  $N$ is a positive integer ($N\geq 1$). Here ${\bm{\delta}}_0$ defines the Dirac measure with pole at $0$,  and $\la>0$ is a small parameter. 
 We assume that $N=1$ and, under some suitable assumptions on the derivatives of the potential $V$  at $0$,  we find a solution which exhibits a non-simple blow-up profile as
$\la\to 0^+$.

\bigskip


\end{abstract}

\maketitle

\section{Introduction}

Given $\Omega$ a smooth and bounded domain in $\R^2$  containing the origin,  consider the following Liouville equation with Dirac mass measure
\begin{equation}\label{eq0}
  \left\{
      \begin{aligned}&- \D u = \la V(x) e^u-4\pi N {\bm{\delta}}_0&  \hbox{ in }&  \Omega,\\
    &  \ u=0 &  \hbox{ on }& \partial \Omega.
  \end{aligned}
    \right. \end{equation}
Here $\la$ is a positive small parameter, the potential $V$ is a positive and smooth function, ${\bm{\delta}}_0$ denotes Dirac mass supported at $0$ and $N $ is a positive integer.

Problem \eqref{eq0} is motivated by its applications in conformal geometry and several fields of physics, where  quite a few semilinear elliptic
equations defined in two dimensional spaces with an exponential nonlinear term are very
commonly observed and studied. The well known prescribing Gauss curvature equation,
mean field equation, Liouville type equations from the Chern-Simons self-dual theory,
and systems of equations of the Toda system are a few examples of this family. The
analysis of these equations is usually challenging as the interesting exponential nonlinear term is always related to the lack of compactness in the variational approach. One
important feature of these equations is the blow-up phenomenon, the understanding of
which is closely related to results on existence, compactness, a-priori estimates, etc.

 The asymptotic behaviour of  family
of blowing up solutions $u_k$
     can be referred to the papers  \cite{breme}, \cite{chenlin0}, \cite{lishafrir}, \cite{mawe},  \cite{nasu}, \cite{su} for the regular problem, i.e. when $N=0$. An extension to the singular case $N>0$
         is contained in \cite{bachenlita}-\cite{bata}.
If a blow-up point $p$ is either a regular point or a ``non-quantized" singular source, the asymptotic behavior of $u_k$ around $p$  is well understood (see \cite{bachenlita,bata,chenlin00,chenlin0,gluck,li,zhang0,zhang}). As a matter of fact, $u_k$ satisfies the spherical Harnack inequality around $0$, which implies that, after scaling, the sequence $u_k$  behaves as a single bubble around the maximum point.  However, if $p$ happens to be a quantized singular source, the so-called ``non-simple" blow-up phenomenon does happen (see \cite{KuoLin, WeiZhang,WeiZhang1, WeiZhang2}), which is equivalent to stating that $u_k$ violates the spherical Harnack inequality around $p$.   The study of non-simple blow-up solutions, whether or not the blow-up point has to be a critical point of coefficient functions,    has been a major challenge for Liouville equations and its research has intrigued people for years. Recently significant progress has been made by Kuo-Lin, Bartolucci-Tarantello and other authors (\cite{bata2,miowei,KuoLin,WeiZhang,WeiZhang1, WeiZhang2}. In particular  it is established in \cite{bata0} and \cite{KuoLin} that  there are $N + 1$ local maximum points and they are evenly distributed on $\mathbb{S}^1$ after scaling according to their magnitude.
  In \cite{WeiZhang} and \cite{WeiZhang1} Harnack inequalities and second order vanishing conditions for non-simple blow-ups are  obtained.

 The case $N\in\mathbb{N}$ is more difficult to treat, and at the same time the most relevant to physical applications. Indeed, in vortex theory the number
$N$
represents vortex multiplicity, so that  in that context the most interesting case is
precisely
when it is a positive integer. The difference between the case $N\in \mathbb{N}$ and $N\not\in \mathbb{N}$ is also analytically essential. Indeed,  as usual in
problems involving  concentration phenomena like \eqref{eq0},  after
suitable rescaling of the blowing-up around a concentration point one sees a limiting equation which, in this case, takes the form of the planar singular Liouville equation:
$$-\Delta U=e^U-4\pi N\delta_0, \quad \int_{\mathbb R^2} e^U dx <\infty;$$ 
 only if $N\in\mathbb{N}$ the above limiting equation   admits non-radial solutions around $0$
 since the family of all solutions extends to one carrying an extra parameter  (see \cite{prata}). This suggests that if $N\in\mathbb{N}$  and the blow-up point happens to be the singular source, then  solutions of \eqref{eq0}  may exhibit non-simple blow-up phenomenon.

So,   from analytical viewpoints the study of non-simple blow-up solutions is far more challenging than simple blow-up solutions, but the impact of this study may be even more significant because they represent certain situations in the blow-up analysis of systems of Liouville equations. Indeed, if local maxima of blow-up solutions in a system tend to one point, the profile of solutions can be described by a Liouville equation with quantized singular source. For all this reasons,  it is desirable to know exactly when non-simple blow-up phenomenon happens. 

However,  the question on the existence of non-simple blowing-up  solutions     to \eqref{eq0} concentrating at $0$     is far from being completely settled. A first definite answer is provided by  \cite{dazhangwei}  which rules out the non-simple blow-up phenomenon for \eqref{eq0} if the potential $V$ is constant: more precisely it is established that there is no non-simple blow-up sequence for \eqref{eq0} with $V=const.$,  even if we are  in the presence of multiples singularities $\sum_i N_i{\bm{\delta}}_{p_i}$.  Apart from this, only partial results are known: in \cite{miowei} the construction of solutions exhibiting a non simple blow-up profile  at $0$  is carried out for equation \eqref{eq0} with $V\equiv 1$ provided that  $\Omega$ is the unit ball and the weight of the source is a positive number $N=N_\la$ close an integer $N$ from the right side.  On the other hand, in\cite{delespomu},  for any fixed positive integer  $N$,   it is proved the existence of  a solution to \eqref{eq0} with $V\equiv 1$, where  ${\bm{\delta}}_0$ is replaced by ${\bm{\delta}}_{p_\la}$ for  a suitable $p_\la\in \Omega$,   with $N+1$ blowing up points  at the vertices of a sufficiently tiny regular polygon  centered in $p_\la$; moreover the location of $p_\la$ is determined by the geometry of the domain in a $\la-$dependent way and does not seem possible to be prescribed arbitrarily.  To our knowledge, the existence of non-simple blow-up phenomenon  for \eqref{eq0} for  a fixed $V$ and a fixed $N$  independent of $\la$ is still open, even in the case of the ball: the only  example is constructed in \cite{mioboh}   for a special class of potentials of the form $V(|x|^{N+1})$.

\

In this paper we investigate the existence of non-simple  blow-up solutions when $\Omega$ is the unit ball $B_1$ centered at the origin, the potential $V_\la=V$ is fixed  and $N=1$:

\begin{equation}\label{eq}
  \left\{
      \begin{aligned}&- \D u = \la V(x)e^u-4\pi  {\mathcal \delta}_0&  \hbox{ in }&  B_1,\\
    &  \ u=0 &  \hbox{ on }& \partial B_1.
  \end{aligned}
    \right. \end{equation}

Let us pass to enumerate the hypotheses on the potential $V$ that will be steadily used throughout the paper. 

\begin{enumerate}

\item[${\mathrm{(H1)}}$] $\inf_{B_1} V(x)>c>0$ for a positive constant $c$ independent of $\la$ and, without loss of generality, we may assume $V(0)=1$;

\
\item[${\mathrm{(H2)}}$] 
$V(x)$  is even, i.e. $$ V(x)=V( -x)\quad \forall x\in B_1.$$ 
\end{enumerate}

Furthermore we will require sufficient regularity of $V$ at $0$ together with   crucial  conditions on the derivatives of $V$ at $0$: \begin{enumerate}
\item[${\mathrm{(H3)}}$] $V(x)$  is of class $C^1$  in the closed unit ball $\overline{B}_1$ and  
the following holds:
\beq\label{expexp}\begin{aligned}V(x)=&\;1+ A_0(x_1^4+x_2^4)+A_1(x_1^3x_2-x_1x_2^3)+ A_2x_1^2x_2^2
\\ &+D_0(x_1^6-x_2^6)+D_1(x_1^5x_2+x_1x_2^5)+D_2(x_1^4x_2^2-x_1^2x_2^4)+D_3x_1^3x_2^3
+O(|x|^7)\end{aligned}\eeq
for some constants $A_0,A_1,A_2,D_0, D_1, D_2, D_3\in \R$.
\end{enumerate}

Let us comment on assumption ${\mathrm{(H3)}}$:  in \cite{WeiZhang}, \cite{WeiZhang1}, \cite{WeiZhang2} the second and the third authors proved that if  non simple blow up scenarios occur for equation \eqref{eq}, then  the first derivatives as well as the Laplacian of coefficient functions must
tend to zero at the singular source;  so 
the vanishing of the second order terms in the expansion \eqref{expexp} is not surprising. Moreover, the analysis reveals that the relation between the forth derivatives and between the sixth derivatives plays a crucial role since it guarantees that the non simple blow-up solutions can be accurately approximated by  global solutions by    allowing an a priori estimate for the error which turns out to be  sufficiently small (see Remark \ref{banale} and Remark \ref{banale1}). 

 In order to provide  the exact formulation of the
result let us fix some notation: in the following $G(x, y)$ is the Green's function of $-\Delta$ over $\Omega$ under Dirichlet boundary conditions and $H(x,y)$ denotes its regular part:
$$H(x,y):=G(x,y)-\frac{1}{2\pi}\log\frac{1}{|x-y|}.$$
In the case of the unit ball we have the explicit formula for the regular part of the Green function in $B_1$ which is given by
\beq\label{accor}H(x,y)=\frac{1}{2\pi}\log\bigg(|x|\bigg|y-\frac{x}{|x|^2}\bigg|\bigg),\quad x,y\in B_1.\eeq 

Then the main result of this paper provides a sufficient condition on the potential $V$, in addition to the assumptions  $\mathrm{(H1)-(H2)-(H3)}$, which implies that \eqref{eq} admits a family of non-simple blowing-up solutions. Such a sufficient condition is expressed in terms of the concept of stable zeroes for a suitable vector field.

\begin{thm}\label{th2} Assume that hypotheses $\mathrm{(H1)-(H2)-(H3)}$ hold and, in addition, \beq\label{techni}A_0=2,\;\; A_1=0\;\; A_2=4.\eeq Let $\xi\in \R^2$, $\xi\neq 0$,  be a zero for the following vector field which is stable under  uniform perturbations\footnote{Given  
$F:\R^2\to\R^2$  a continuous vector filed, we say that $\xi$ is  a zero for $F$ which is stable with respect to uniform perturbations if $F(\xi)=0$ and for any neighborhood $U$ of $\xi$ and  $\epsilon>0$ there exists $\eta>0$ such that if $\Psi:U\to \R^2$ is continuous and $\|\Psi-F\|_\infty\leq \eta$, then $\Psi$ has a zero in $U$. A sufficient condition which implies that $0$ is a stable zero of a vector field $F$ is
$deg(F, U, 0)\neq 0$ for some neighborhood $U$ of $\xi$, where $deg$ denotes the standard Brower degree. } 
\beq\label{vector}(\xi_1,\xi_2)\longmapsto 
\left(\begin{aligned} & 3D_0\xi_1^2+D_1\xi_1\xi_2+\frac{3D_0-D_2}{4}\xi_2^2+\frac{15D_0-D_2}{4}\\ & 
\frac{D_1}{2}\xi_1^2+\frac{3D_0-D_2}{2}\xi_1\xi_2+3\frac{2D_1+D_3}{8}\xi_2^2+\frac{10D_1+3D_3}{8}\end{aligned}\right)
.\eeq 
 Then, for $\la$ sufficiently small
the problem \eqref{eq} has a family of solutions $u_\la$  blowing up at the origin as $\la\to 0^+$:
$$\la e^{u_\la}  \to 16\pi {\bm{\delta}}_0\;\;\hbox{ in the measure sense.}$$
 More precisely there exist $\de=\de(\la)>0$ and
$b=b(\la)\in B_1$ in a neighborhood of $0$ such that $u_\la$ satisfies
$$u_\la+4\pi  G(x,0)=-2\log\big(\mu^{4}+|x^{2}-b|^{2}\big)+8\pi H(x^{2},b)+o(1)$$  in $H^1$-sense, where 
 \beq\label{saltria}b(\la) =\frac{\xi_0}{4\sqrt 2}\sqrt{\la \log\frac{1}{\la}}\big(1+o(1)\big), \qquad \mu^2(\la)= \frac{\sqrt\la}{4\sqrt2} (1+o(1)) .\eeq
In particular,   $\mu^2=o(|b|).$
\end{thm}
The solution constructed in Theorem \ref{th2} reveals a non-simple blow-up profile: indeed, denoting by $\pm\beta $ the square complex roots of $b$,  since the rate of convergence $\beta\to 0$ is lower than the speed of the concentration parameter $\mu\to 0$ (see estimate \eqref{saltria}), then $u_\la$ develops  $2$  local maximum points  concentrating at $0$
 which are arranged 
 close
 to two opposite vertices. 
 The analysis shows that the configuration of the limiting local maxima  is
determined by the interaction of two crucial aspects: the effect of the potential $V$, which tends to shrink the bubble to $0$, and the boundary effect,
represented by the Robin function $H(\xi,\xi)$, which tends to repel the bubble from $0$. On the other hand,  the existence of this kind of non-simple blow-up is still open for more general potential $V$. Indeed, as we will observe in Remark \ref{banale1}, if we apply our method for generic values $A_0, A_1, A_2$ not satisfying \eqref{techni}, then  we
find out that the forces exerted between the potential and the boundary may not balance and we
are unable to catch a solution different from the radially symmetric one.

\begin{rmk}\label{potential}  Let us observe that $F$ actually corresponds to a gradient field, precisely $F(\xi)=\nabla J(\xi)$, where the potential $J$ is given by
$$J(\xi)=D_0 \xi_1^3+\frac{D_1}{2}\xi_1^2\xi_2+\frac{3D_0-D_2}{4}\xi_1\xi_2^2+ \frac{2D_1+D_3}{8}\xi_2^3+   \frac{15D_0-D_2}{4}\xi_1+\frac{10D_1+3D_3}{8}\xi_2.$$
\end{rmk}

\begin{ex} Let us provide  explicit examples of coefficients $D_0, D_1, D_2$ for which $0$ is a stable zero for $F$, so that, according to Theorem \ref{th2} the corresponding $V$ will produce a non simple blowing up solution for eqation \eqref{eq}. Indeed,
if we take $D_0=0,$ $D_1=2$, $D_2=-4$, $D_3=-4$,     then the potential $J$ defined in Remark \ref{potential} becomes
$$J(\xi)=\xi_1^2\xi_2 +\xi_1\xi_2^2 +\xi_1+\xi_2.$$
It is immediate to check that $(1,-1)$ (respectively $(-1,1)$) is a  critical points for $J$, so $$F(1,-1)=\nabla J(1,-1)=0$$
Moreover, the Hessian matrix of $J$ at $(1,-1)$ is given by $$\left(\begin{aligned}-&2&0&\\ &0&2&\end{aligned}\right).$$
Then $(1,-1)$ (respectively $(-1,1)$) turns out to be a nondegenerate critical point for $J$ of saddle type, so $deg(F, U, 0)\neq 0$ where $deg$ denotes the standard Brower degree and $U$ is a sufficiently small neighbourhood of $(1,-1)$. Consequently, $(1,-1)$ is a stable (with respect to uniform perturbations) zero for $F$. 
Then, according to Theorem \ref{th2}, 
if $V$ satisfies $\mathrm{(H1)-(H2)}$ and 
$$V(x)=1+ 2x_1^2+4x_1^2x_2^2+2x_2^4+2(x_1^5x_2+x_1x_2^5)+4(x_1^4x_2^2-x_1^2x_2^4)-4x_1^3x_2^3
+O(|x|^7)$$
then
$\xi_0:=(1,-1)$ (respectively $\xi_0:=(-1,1)$) gives rise to a non-simple blowing up family of solutions to \eqref{eq}.
\end{ex}

\

\
The phenomena of non-simple bubbling solutions not only occur in single equations, but also in systems. In a recent work of the second author and Gu (\cite{Gu}) the
non-simple blow-up behaviours  are studied for singular Liouville systems. In another work of the second, third authors and Wu \cite{wei-wu-zhang-1} non-simple blowup is ruled out for Toda systems. Examples of non-simple blow-up solutions are available for other models: we
recall, for instance, the Liouville equation with anisotropic coefficients in \cite{weiyezhou} and  the Toda system in \cite{weiwei}.

\

The proofs use singular perturbation methods which combine the variational approach with a Lyapunov-Schmidt type procedure. Roughly speaking, the first step consists in the construction of an approximate solution, which should turn out to  be precise enough. In view of the expected asymptotic behaviour,  the shape of such  approximate solution will resemble, after the change of varia\-bles $x\mapsto x^{1/2}$, a \textit{bubble} of the form \eqref{bubble} with  a suitable choice of the parameter $\de=\de(\la,b)$. Then  we look for a solution to \eqref{eq} in a small neighborhood of the first approximation.  As quite standard in singular perturbation theory, a crucial ingredient is nondegeneracy of
the explicit family of
solutions of the limiting Liouville problem \eqref{limit}, 
as established in \cite{bapa}.
This allows us to study the invertibility of the linearized operator associated to the problem \eqref{eq}  under suitable orthogonality conditions. Next we introduce an intermediate problem and  a fixed point argument will provide a solution for an  auxiliary equation, which turns out to be solvable for any choice of $b$. Finally we test the auxiliary equation on the elements of the kernel of the linearized operator  and we find out that,
 in order to find an \textit{exact} solution of \eqref{eq}, the location of the maximum points,  which is detected by the parameter $b$, should be a zero for a reduced finite dimensional map. The  main technical difficulty in the proof is that we need to expand the reduced map up to higher orders to catch a nontrivial zero $b$, which will give rise to a non-simple blow-up  solution. Moreover the method fails for $N\geq2$: indeed if we try to apply our technique to $N\geq 2$, then the analogous of assumption $\mathrm{(H3)}$ would give that the potential has vanishing derivatives up to the order $N+1$ at $0$, and this implies that the approximation rate  for the reduced finite dimensional map is unfortunately not sufficiently small to carry out the argument.

\

The rest of the paper is organized as follows. Section 2 is devoted to some preliminary
results, notation, and the definition of the approximating solution. Moreover, a more
general version of Theorems \ref{th2} is stated there (see Theorem \ref{main2}).  In Section 3  we
sketch the solvability of the linearized problem by referring to \cite{delkomu} and \cite {espogropi} for the proof.  The error up to which the approximating solution solves problem \eqref{eq} is estimated in Section 4.   Section 5 considers the solvability of an auxiliary problem by a  contraction argument. In Section 6  we
complete the proof of Theorem \ref{th2}. In Appendix A we collect some results, most of them well-known,  which are usually referred to throughout the
paper.

\

\noindent{\bf{NOTATION}}: In our estimates throughout the paper, we will frequently denote by $C>0$, $c>0$ fixed
constants, that may change from line to line, but are always
independent of the variables under consideration. 

\section{Preliminaries and statement of the main results}

We are going to provide an equivalent formulation of problem \eqref{eq} and Theorem \ref{th2}. Indeed, 
setting  $v$ the regular part of $u$, namely \beq\label{chva}v= u+4\pi  G(x, 0)= u+2\log\frac{1}{|x|},\eeq  problem \eqref{eq} is then equivalent to solving the following  boundary value problem
\beq\label{proreg}\left\{\begin{aligned} & -\Delta v=\la  |x|^{2}V(x)e^v&\hbox{ in }& B_1\\ &v=0&\hbox{ on }&\partial B_1\end{aligned}\right..\eeq
Here $G$ and $H$ are the Green's function and its regular part as defined in the introduction.

 In what follows, we identify $x=(  x_1,x_2)\in \R^2$ with $x_1+{\rm i}x_2\in \C$ and   we denote by $x\,y$ the multiplication of the complex numbers $x,y$ and, analogously,  by $x^2$ the square of the complex number $x$.

Since  $V$ and the solutions  considered in the paper are
even, we can rewrite problem  \eqref{proreg} as a regular Liouville problem: more precisely, denoting by $x^{\frac12}$ the complex $2$-roots of $x$, the change of variables \beq\label{chva1}w(x)= v\big(x^{\frac{1}{2}}\big)\eeq transforms problem \eqref{proreg} into a (regular) Liouville problem of the form 

\beq\label{proreg1}\left\{\begin{aligned}&-\Delta w= \frac{\la}{4} V \big(x^{\frac12}\big)e^w&\hbox{ in }&B_1\\ &w=0&\hbox{ on }&\partial B_1\end{aligned}\right..\eeq

Theorem \ref{th2} will be a consequence of  a more general result concerning Liouville-type problems. In order to provide such  a result, we now give a construction of a suitable approximate solution for \eqref{proreg1}.
We can associate to \eqref{proreg1} a limiting problem of Liouville type which will play a crucial role in the construction of blowing up solutions  as $\la\to 0^+$:  \beq\label{limit}
-\Delta W=e^W\quad \hbox{in}\;\; \rr^2,\qquad
\int_{\R^2} e^{W(x)}dx<+\infty.
\eeq
All solutions of this problem  are given, in complex notation,  by the three-parameter family of functions
\beq\label{bubble}
W_{\de,b}(x):=\log  \frac{8\de^{2}}{ (\de^{2}+|x-  b|^2)^2}\quad
\de>0,\,b\in \C.
\eeq
The following quantization property holds: \beq \label{quantum} \int_{\R^2}
e^{W_{\de,b}(x)}dx = 8 \pi  .\eeq
In the following we agree that
$$W_{\la}(x)=W_{\de,b}(x),\quad \de>0,\,b\in \C,$$ where the value $\delta=\delta(\la,b)$ is defined by

\begin{equation} \label{delta} \delta^{{2}}:=\frac{\la}{32}{V}(b^{\frac12})e^{8\pi H(b,b)} =\frac{\la}{32}{V}(b^{\frac12})(1-|b|^2)^4. \end{equation}

We point out that the diagonal $H(b,b) $ appearing in \eqref{delta} is called the Robin function of the domain and in the case of the ball it takes the form$$H(x,x)=\frac{1}{2\pi}\log(1-|x|^2),\quad x\in B_1$$ according to \eqref{accor}. 
To obtain a better first approximation, we need to modify the function $W_{\la}$   in order to satisfy the zero boundary condition. Precisely, we consider the projection $P W_{\la} $ onto the space $ H^1_{0}( B_1 )$, where the projection  $P:H^1(\R^N)\to  H^1_{0}( B_1 )$ is
defined as the unique solution of the problem
$$
 \Delta P v=\Delta v\quad \hbox{in}\  B_1 ,\qquad  P v=0\quad \hbox{on}\ \partial B_1 .
$$
We recall   that  the regular part $H(x,b)$ of the Green function, defined in \eqref{accor},  is harmonic in $B_1$ and satisfies $ H (x,b)=\frac{1}{2\pi}\log|x-b|$   for $x\in\partial  B_1 ;$ a straightforward computation gives that for any 
 $x\in\partial B_1 $ 
$$\begin{aligned}PW_{\la}- W_{\la}+\log\(8\de^{2}\)-8\pi H(x,b)&=- W_{\la}+\log\(8\de^{2}\)-4\log|x-b|
=2\log\Big(1+\frac{\de^2}{|x-b|^2}\Big)
\\ &= 2\frac{\de^2}{|x-b|^2}+O(\de^4) = 2\frac{\de^2}{1+O(|b|)}+O(\de^4)\\ &= 2\de^2+O(\de^2|b|)+O(\de^4)
\end{aligned}$$
with uniform estimate for  $x\in\partial B_1$ and $b$ in a small neighborhood of $0$.
Since the  expressions $PW_{\la}- W_{\la}+\log\(8\de^{2}\)-8\pi H(x,b)$  and $2\de^2$ are    harmonic in $ B_1 $, then the maximum principle applies and implies
the following asymptotic expansion
\beq\label{pro-exp1}\begin{aligned}
 PW_{\la}=& W_{\la}-\log\(8\de^{2}\)+8\pi H (x,b)+2\de^2+O(\de^2|b|)+O(\de^4)\\ & 
 =-2\log\(\de^{{2}}+|x-b|^2\)+8\pi H (x,b)+2\de^2+O(\de^2|b|)+O(\de^4)
\end{aligned}\eeq
uniformly for $x\in \overline B_1  $ and $b$ in a small neighborhood of $0$.

We point out that, in order to simplify the notation,  in our estimates throughout the paper we will describe the asymptotic behaviors of quantities under considerations in terms of $\de=\de(\la, b)$ instead of the parameter $\la$ of the equation. Clearly according to \eqref{delta} $\delta$ has the same rate as $\la^{\frac{1}{2}} $, so at each step  we can easily pass to the analogous asymptotic in terms of $\la$: for instance, in \eqref{pro-exp1} the  error term ``$O(\de^{4})$" can be equivalently replaced by ``$O(\la^2)$".

We shall look for a solution to \eqref{proreg1} in a small neighborhood of the first approximation, namely a solution of the form
 $$w_\la=PW_{\la}
 + {\phi}_\la,$$ where the rest term
$\phi_\la$ is small in
$H^1_0( B_1 )$-norm.

Let us reformulate the main theorem for problem \eqref{proreg1}, which prove that a  non symmetric blow-up occurs for problem \eqref{proreg1}. 
More precisely, we provide a solution which develops a bubble centered at a point $b$; and since the rate of convergence $b\to 0^+$ is lower than the speed of the concentration parameter $\de\to 0^+$ (see estimate \eqref{saltria11}), then the blowing up turns out to be non symmetric in the first approximation.

\begin{thm} \label{main2} 
Assume that hypotheses  ${\mathrm{(H1)-(H3)}}$ and \eqref{techni} hold.  Let $b\in \R^2$ be a zero for the vector field \eqref{vector} which is stable under uniform perturbations. 
  Then, for $\la$ sufficiently small
the problem \eqref{proreg1} has a family of solutions $w_\la$  satisfying
$$w_\la=-2\log\big(\de^{2 }+|x-b_\la|^{2}\big)+8\pi H (x,b_\la)+o(1)$$  in $H^1$-sense, where 
  \beq\label{saltria11}b_\la =\frac{\xi_0}{4\sqrt 2}\sqrt{\la \log\frac{1}{\la}}\big(1+o(1)\big).\eeq
In particular,  by \eqref{delta},  $\de^2=o(|b_\la|).$
\end{thm}

In the remaining part of this paper we will prove Theorems \ref{main2} and at the end of Section 6 we shall see how Theorems \ref{th2} follows quite directly as a corollary. 

We end this section by setting  notation and basic well-known
facts which will be of use in the rest of the paper. Given $\Omega$ a bounded domain, we  denote by  $\|\cdot\|$ and $\|\cdot\|_p$  the norms in  the space $ H^1_0(\Omega)$ and $L^p(\Omega)$, respectively, namely
 $$\|u\|:=\|u\|_{ H^1_0(\Omega)}
 ,\qquad \|u\|_p:=\|u\|_{L^p(\Omega)}
 \quad \forall u\in  H^1_0(\Omega).$$

In next lemma we recall the well-known Moser-Trudinger inequality
(\cite{Moe, Tru}).

\begin{lemma}\label{tmt} There exists $C>0$ such that for any bounded domain $\Omega$ in $\rr^2$
 $$\int_\Omega e^{\frac{4\pi u^2}{\|u\|^2}}dy\le C |\Omega|\quad \forall u\in{ H}^1_0(\Omega),$$ where  $|\Omega|$ stands for the measure of the domain $\Omega$.
 In particular,  for any $q\geq 1$
 $$\| e^{u}\|_{q}\le  C^{\frac1q} |\Omega|^{\frac1q} e^{{\frac{q}{ 16\pi}}\|u\|^2}\quad \forall  u\in{H}^1_0(\Omega).$$

\end{lemma}

As commented in the introduction, our proof uses the singular
perturbation methods. For that, the nondegeneracy of the functions
that we use to build our approximating solution is essential. Next
proposition is devoted to the nondegeneracy of the finite mass
solutions of the Liouville equation (see \cite{bapa} for the proof).

\begin{prop}
\label{esposito} Assume that $\xi\in\R^2$ and $\phi:\R^2\to\R$  solves the problem
\begin{equation}\label{l1}
-\Delta \phi =\frac{8}{(1+|z-\xi|^2)^2}\phi\;\;
\hbox{in}\ \rr^2,\quad \int_{\R^2}|\nabla
\phi(z)|^2dz<+\infty.
\end{equation}
 Then there exist $c_0,\,c_1,\, c_2\in\R$
such that
$$\phi(z)=c_0  Z_0+ c_1Z_1 +c_2Z_2,$$
$$Z_0(z):   = {1-|z-\xi|^2\over 1+|z-\xi|^2} ,\ \; \;Z_1(z):={ z_1-\xi_1\over  1+|z-\xi|^2}
,\ \;\;Z_2(z):={ z_2-\xi_2 \over  1+|z-\xi|^2}.
$$ \end{prop}

\section{Analysis of the linearized operator}
According to Proposition \ref{esposito}, by the change of variable $x=\de z$, we immediately get that  all solutions $\psi$  of
$$
-\Delta \psi= {8\de^{2}\over (\de^{2}+|x-b|^2)^2}\psi =e^{W_{\la}}\psi\quad \hbox{in}\quad \rr^2,\qquad \int_{\R^2}|\nabla
\phi(x)|^2dx<+\infty.$$
are linear combinations of the functions
$$Z^0_{\delta,b}(x)={\delta^{2}-|x-b|^2\over \delta^{2}+|x-b|^2},\ Z^1_{\delta,b}(x)=
{ \de(x_1-b_1)\over  \de^{2}+|x-b|^2},\ Z^2_{\delta,b}(x)=
{ \de (x_2-b_2)\over  \de^{2}+|x-b|^2}
.$$
We introduce the projections $PZ^j_{\delta,b}$ onto $H^1_0( B_1 )$.  It is immediate that
 \begin{equation}\label{pz0}
PZ^0_{\delta,b}(x)=Z^0_{\delta,b}(x)+1+ O\(\de^{2}\)={ 2 \de^{2} \over \de^{2}+|x-b|^{2} }+ O(\delta^{2})
\end{equation}
and
  \begin{equation}\label{pzi}
PZ^j_{\delta,b}(x)=Z^j _{\delta,b}(x) +O(\de)\hbox{ for }j=1,2\end{equation}
uniformly with respect to $x\in\overline B_1 $ and $b> 0$ in a small neighborhood of $0.$

We agree that $Z_\la^j:=Z_{\delta,b}^j$ for any $j=0,1,2$, where $\delta$ is defined in terms of  $\la$ and $b$ according to \eqref{delta}.
Let us consider the following linear problem: given
$ h\in H^1_{0}( B_1 )$,  find a function $\phi\in  H^1_{0}( B_1 )$  and  constant $c_1,c_2\in\R$ satisfying
\begin{equation}\label{lla2}
\left\{\begin{aligned}&-\Delta \phi   -\frac{\la}{4} V\big(x^\frac12\big) e^{P{W}_\la}\phi=\Delta h+\sum_{j=1,2}c_j Z^j_{\la} e^{W_{\la}}\\ &\into \nabla \phi\nabla PZ^j_{\la}=0\;\;j=1,2\end{aligned}\right..
\end{equation}

In order to solve problem \eqref{lla2}, we need to establish an a priori estimate. For the proof we refer to \cite{delkomu} (Proposition 3.1) or \cite{espogropi} (Proposition 3.1). 
\begin{prop}\label{linear2}
There exist $\lambda_0>0$ and $C>0$ such that for any $\la \in(0, \la_0)$, any $b$ in a small neighborhood of $0$
and   any $h\in  H^1_{0}( B_1 )$, if
$(\phi,c_1,c_2)\in  H^1_{0}( B_1 )\times \R^2$  solves \eqref{lla2}, then the following holds $$\|\phi\|    \leq C  |\log\de |    \|h\| .$$ \end{prop}

 For any $ p>1,$ let \beq\label{istar}i^*_{p}:L^{p}( B_1 )\to H^1_{0}( B_1 )\eeq be the
adjoint operator of the embedding
$i_{p}:H^1_{0}( B_1 )\hookrightarrow L^{p\over p-1 }( B_1 ),$ i.e.
$u=i^*_{p}(v)$ if and only if $-\Delta u=v$ in $ B_1 ,$ $u=0$ on
$\partial B_1 .$ We point out that $i^*_{p}$ is a continuous
mapping, namely
\begin{equation}
\label{isp} \|i^*_{p}(v)\| \le  c_p \|v\|_{p}, \ \hbox{for any} \ v\in L^{p}( B_1 ),
\end{equation}
for some constant $c_{p}$ which depends on  $p.$
Next let us set
$$      {K }:=\hbox{span}\left\{PZ^1_{\la}, \,PZ_\la^2\right\}$$
and
$$      {K ^\perp }:= \left\{\phi\in H^1_{0}( B_1 )\ :\ \into \nabla \phi \nabla PZ^j_{\la}dx=0\;\;j=1,2\right\} $$
and  denote by
$$     \Pi : H^1_{0}( B_1 )\to       {K },\qquad      {\Pi ^\perp}: H^1_{0}( B_1 )\to       {K ^\perp }$$
the corresponding projections.
Let $      L: K^\perp \to K^\perp$  be the
linear operator defined by
\beq\label{elle}
      L(     \phi):= \frac14 \Pi^{\perp}\Big(   {i^*_{p}}\big(       \la V\big(x^\frac12\big)e^{PW_{\la}}     \phi \big)\Big) - \phi.
\eeq

Notice that problem \eqref{lla2} reduces to $$L(\phi)=\Pi^\perp h, \quad \phi\in K^\perp.$$

 As a consequence of Proposition \ref{linear2} we derive the invertibility of $L$.

\begin{prop}\label{ex} For any $p>1$ there exist $\lambda_0>0$ and $C>0$ such that for any $\la \in(0, \la_0)$, any $b$ in a small neighborhood of $0$ and
 any $h\in K^\perp$  there is a unique solution $ \phi\in K^\perp$ to the problem $$L(\phi)=h.$$ In
particular, $L$ is invertible; moreover, $$\| L^{-1} \| \leq C
|\log \de |.$$

\end{prop}
\begin{proof}  Observe that the operator $\phi\mapsto \Pi^\perp\big( {i^*_{p}}(       \la V\big(x^\frac12\big)e^{PW_{\la}}     \phi )\big)$ is a compact operator in $K^\perp$.
   Let us consider the case $h=0$, and take $\phi\in K^\perp$ with $L(\phi)=0$. In other words,  $\phi$ solves
the system \eqref{lla2} with $h=0$ for some $c_1,c_2\in\R$. Proposition \ref{linear2} implies $\phi\equiv 0$. Then, Fredholm's alternative implies the existence and uniqueness result.

Once we have existence, the norm estimate follows directly from Proposition \ref{linear2}.
\end{proof}

\section{Estimate of the error term}

The goal of this section is to provide an estimate of the error up to which the approximate
solution $PW_\la$ solves problem \eqref{proreg1}. First of all, we perform the following estimates.

\begin{lemma}\label{aux00}   Let
 $\gamma=0,1,2$ and $p>1$ be fixed. The following holds:\beq\label{aux00esti1}
\| |x-b|^\gamma e^{W_{\la}}\|_p
\leq C\de^{\gamma}\de^{-2\frac{p-1}{p}} ,\quad \| |x-b|^\gamma \la e^{PW_{\la}}\|_p
\leq C\de^{\gamma}\de^{-2\frac{p-1}{p}} \eeq
uniformly for $b$ in a small neighborhood of $0$.
\end{lemma}
\begin{proof}
We compute $$\begin{aligned}\||x-b|^\gamma e^{W_{\la}}\|_p^p&=
8^p\de^{2 p}\into\frac{|x-b|^{\gamma p}}{(\de^{2}+|x-b|^2)^{2p}}dx
\leq 8^p\de^{ \gamma p-2(p-1)} \intr\frac{|z|^{\gamma p}}{(1+|z|^2)^{2p}}dz
.\end{aligned}$$
Taking into account that the last integral is finite for   $\gamma=0,1,2$ and $p>1$ 
we deduce the first part of  \eqref{aux00esti1}.
To prove the second part it is sufficient to observe that 
by \eqref{pro-exp1} and by the choice of $\delta$ in \eqref{delta} we derive
\beq\label{poi}\la e^{PW_{\la}}= \frac{\la}{8\de^{2}}e^{W_{\la}+O(1)}
=e^{W_{\la}}(1+O(1)).\eeq
\end{proof}
\begin{lemma}\label{crucru}      Assume that hypotheses ${\mathrm{(H1)-(H3)}}$ hold. There exists  $P :\R^2\to\R$ a homogeneous polynomial of degree 2 such that 
$$\begin{aligned}\frac{V(x^\frac12)}{V(b^\frac12)}&=1+ 2A_0b_1(x_1-b_1)+\frac{A_1}{2}\Big(b_1(x_2-b_2)+b_2(x_1-b_1)\Big)+\Big(A_0+\frac{A_2}{2}\Big)b_2(x_2-b_2)
\\ &\;\;\;\;+ D_0\Big((x_1-b_1)^3+3b_1^2(x_1-b_1)\Big)+\Big(\frac{D_1}{4}+\frac{D_3}{8}\Big)\Big((x_2-b_2)^3+3b_2^2(x_2-b_2)\Big)\\ &\;\;\;\;+\frac{D_1}{2}\Big((x_1-b_1)^2(x_2-b_2)+b_1^2(x_2-b_2)+2b_1b_2(x_1-b_1)\Big)
\\ &\;\;\;\;+\frac14\Big(3D_0-D_2\Big)\Big((x_1-b_1)(x_2-b_2)^2+b_2^2(x_1-b_1)+2b_1b_2(x_2-b_2)\Big)
\\ &\;\;\;\;+P(x-b)+O(|b||x-b|^2)+O(|b|^3|x-b|)
+O(|x-b|^{\frac72})+O(|b|^{\frac72})
\end{aligned}$$ uniformly for  $b$ in a small neighborhood of $0$.
\end{lemma}
\begin{proof} Let us first consider a more general potential $V$ of the form 
$$\begin{aligned}V(x)&=1+
\sum_{j=0}^{4}A_j x_1^{4-j}  x_2^j+\sum_{j=0}^{6}D_jx_1^{6-j}x_2^j  +O(|x|^{7}),\quad A_j,D_j\in\R,
\end{aligned}$$
and, using the polar coodinates $x=\rho e^{{\rm i}\theta}=(\rho \cos\theta, \rho\sin\theta)$, we have
$$\begin{aligned}V(x^\frac12)&=1+
\rho^{2}\sum_{j=0}^{4}A_j\cos^{4-j}  \frac\theta2 \sin^j\frac\theta2+\rho^{3}\sum_{j=0}^{6}D_j\cos^{6-j} \frac\theta2 \sin^j \frac\theta2 +O(|x|^{\frac{7}{2}}).
\end{aligned}$$ 
Now we use standard trigonometric identities to obtain:
$$\cos^4\frac\theta2=
\frac{\sin^2\theta +2\cos\theta+2\cos^2\theta}{4},\qquad \sin^4\frac\theta2
= \frac{\sin^2\theta-2\cos\theta+2\cos^2\theta}{4} $$ $$\cos^6\frac\theta2
= \frac{1+4\cos^3\theta+3\cos\theta\sin^2\theta +3\cos^2\theta}{8}, \quad \sin^6\frac\theta2=
\frac{1-4\cos^3\theta-3\sin^2\theta\cos\theta+3\cos^2\theta }{8}$$
$$\cos \frac\theta2 \sin^3\frac\theta2=\sin\theta \frac{1-\cos\theta}{4},\qquad \sin \frac\theta2 \cos^3\frac\theta2=\sin\theta \frac{1+\cos\theta}{4}   ,\qquad \sin^2\frac\theta2\cos^2\frac\theta2=\frac14\sin^2\theta,$$
$$\cos^5\frac\theta2\sin\frac\theta2=
\sin\theta \frac{2\cos^2\theta +\sin^2\theta +2\cos\theta}{8} ,\qquad \sin^5\frac\theta2\cos\frac\theta2=
\sin\theta \frac{2\cos^2\theta+\sin^2\theta -2\cos\theta}{8},$$
$$\sin^3\frac\theta2\sin^3\frac\theta2= \frac18\sin^3\theta,\qquad \cos^2\frac\theta2\sin^4\frac\theta2=\sin^2\theta\frac{1-\cos\theta}{8},\qquad \sin^2\frac\theta2\cos^4\frac\theta2=\sin^2\theta\frac{1+\cos\theta}{8}.$$
According to  ${\mathrm{(H3)}}$  we get \beq\label{relat}A_0=A_4,\;\; A_1=-A_3,\;\; D_0=-D_6,\;\;D_1=D_5,\;\; D_4=-D_2,\eeq
so we derive
\beq\label{deep}\begin{aligned}V(x^\frac12)&= 1+A_0\Big(x_1^2+\frac{x_2^2}{2}\Big) +  \frac{A_1}{2}x_1x_2+\frac{A_2}{4}x_2^2\\ & \;\;\;+D_0\Big(x_1^3+\frac34 x_1x_2^2\Big)+D_1\Big(\frac{x_1^2x_2}{2}+\frac{x_2^3}{4}\Big) -\frac{D_2}{4}x_1x_2^2+\frac{D_3}{8}x_2^3
+ O(|x|^{3 +\frac{1}{2}})\\ &=
1+A_0x_1^2 +  \frac{A_1}{2}x_1x_2+\Big(\frac{A_0}{2}+\frac{A_2}{4}\Big)x_2^2\\ &\;\;\;+ D_0x_1^3+\frac{D_1}{2}x_1^2x_2+\frac14\Big(3D_0-D_2\Big)x_1x_2^2+\Big(\frac{D_1}{4}+\frac{D_3}{8}\Big)x_2^3+ O(|x|^{3 +\frac{1}{2}}).
\end{aligned}\eeq

Next observe that, setting $x_1=(x_1-b_1)+b_1$  and $x_2=(x_2-b_2)+b_2$ and making trivial computations we get   $$\begin{aligned}&A_0x_1^2 +  \frac{A_1}{2}x_1x_2+\Big(\frac{A_0}{2}+\frac{A_2}{4}\Big)x_2^2\\ &\;\;\;+ D_0x_1^3+\frac{D_1}{2}x_1^2x_2+\frac14\Big(3D_0-D_2\Big)x_1x_2^2+\Big(\frac{D_1}{4}+\frac{D_3}{8}\Big)x_2^3\\ &=2A_0b_1(x_1-b_1)+\frac{A_1}{2}\Big(b_1(x_2-b_2)+b_2(x_1-b_1)\Big)+\Big(A_0+\frac{A_2}{2}\Big)b_2(x_2-b_2)
\\ &\;\;\;+ D_0\Big((x_1-b_1)^3+3b_1^2(x_1-b_1)\Big)+\frac{D_1}{2}\Big((x_1-b_1)^2(x_2-b_2)+b_1^2(x_2-b_2)+2b_1b_2(x_1-b_1)\Big)
\\ &\;\;\;+\frac14\Big(3D_0-D_2\Big)\Big((x_1-b_1)(x_2-b_2)^2+b_2^2(x_1-b_1)+2b_1b_2(x_2-b_2)\Big)
\\ &\;\;\;+\Big(\frac{D_1}{4}+\frac{D_3}{8}\Big)\Big((x_2-b_2)^3+3b_2^2(x_2-b_2)\Big)
\\ &\;\;\;
+A_0b_1^2+\frac{A_1}{2}b_1b_2+\Big(\frac{A_0}{2}+\frac{A_2}{4}\Big)b_2^2+D_0b_1^3+ \frac{D_1}{2}b_1^2b_2+\frac14\Big(3D_0-D_2\Big)b_1b_2^2+\Big(\frac{D_1}{4}+\frac{D_3}{8}\Big)b_2^3\\ &\;\;\;+P(x-b)+O(|b||x-b|^2).\end{aligned}
$$
Finally, since by \eqref{deep} $$\begin{aligned}V(b^\frac12)=&\;\,1+ A_0b_1^2+\frac{A_1}{2}b_1b_2+\Big(\frac{A_0}{2}+\frac{A_2}{4}\Big)b_2^2\\ &+D_0b_1^3+ \frac{D_1}{2}b_1^2b_2+\frac14\Big(3D_0-D_2\Big)b_1b_2^2+\Big(\frac{D_1}{4}+\frac{D_3}{8}\Big)b_2^3
 +O(|b|^{3+\frac12})\end{aligned}$$ and, consequently, $\frac{1}{V(b^\frac12)}=1+O(|b|^2) 
$, substituting into \eqref{deep} we obtain the thesis. 

\end{proof}

\begin{rmk}\label{banale000} Let us observe that thanks to the symmetry of the coefficients  \eqref{relat} we obtain that $V(x^\frac12)$ turns out to be three times differentiable at $0$: 
indeed the choice of coefficients  implies that the two sums $\sum_{j=0}^{4}A_j\cos^{4-j}  \frac\theta2 \sin^j\frac\theta2$ and $\sum_{j=0}^{6}D_j\cos^{6-j} \frac\theta2 \sin^j \frac\theta2$ turn out to be polynomials in the variables $\cos\theta,$ $ \sin\theta$ of degree 2 and 3 respectively. 
\end{rmk}

Now we are in the position to provide the error estimate.

\begin{prop}\label{aux000}  Assume that hypotheses ${\mathrm{(H1)-(H2)-(H3)}}$ and \eqref{techni} hold
and define
$$\begin{aligned}{R}_\la:=&\frac{\la}{4}  V \big(x^{\frac12}\big) e^{PW_{\la}}
+\Delta PW_{\la}= \frac{\la}{4}  V \big(x^{\frac12}\big) e^{PW_{\la}}-e^{W_\la}.
\end{aligned}$$ Then  the following holds 
\beq\label{stimarafff}\begin{aligned}R_\la=&\;2\de^2e^{W_\la}+ D_0e^{W_\la}\Big((x_1-b_1)^3+3b_1^2(x_1-b_1)\Big)\\ &+\frac{D_1}{2}e^{W_\la}\Big((x_1-b_1)^2(x_2-b_2)+b_1^2(x_2-b_2)+2b_1b_2(x_1-b_1)\Big)
\\ &+\frac14\Big(3D_0-D_2\Big)e^{W_\la}\Big((x_1-b_1)(x_2-b_2)^2+b_2^2(x_1-b_1)+2b_1b_2(x_2-b_2)\Big)
\\ &+\Big(\frac{B_1}{4}+\frac{D_3}{8}\Big)e^{W_\la}\Big((x_2-b_2)^3+3b_2^2(x_2-b_2)\Big)
+P(x-b)e^{W_\la}\\ &+O(\de^2|x-b|)e^{W_\la}+ O(|b||x-b|^2)e^{W_\la}+O(|b|^3|x-b|)e^{W_\la}
+O(|x-b|^{\frac72})e^{W_\la}\\ &+O(|b|^{\frac72})e^{W_\la}
+O(\de^2|b|)e^{W_\la}+O(\de^4)e^{W_\la}\end{aligned}\eeq 
uniformly for  $b$ in a small neighborhood of $0$.   
Moreover  for any $p>1$
$$\|R_\la\|_p\leq C (\de^2+|b|^3)\de^{-2\frac{p-1}{p}} 
$$ 
uniformly for  $b$ in a small neighborhood of $0$.

\end{prop}
\begin{proof} 
By \eqref{pro-exp1} and the choice of $\delta$ in \eqref{delta} we derive
\beq\label{proprop}\begin{aligned}\frac{\la}{4} V(x^\frac12)e^{PW_{\la}}&= \frac{\la}{32\de^{2}} V (x^\frac12)e^{W_{\la}+8\pi H (x,b)+2\de^2+O(\de^2|b|)+O(\de^4)} 
\\ &
=\frac{V(x^\frac12)}{ V (b^\frac12)}e^{W_{\la}}e^{8\pi (H(x, b)-H(b,b))+2\de^2+O(\de^2|b|)+O(\de^4)} 
\\ &=\frac{V(x^\frac12)}{ V (b^\frac12)}e^{W_{\la}}e^{8\pi (H(x, b)-H(b,b))}\big(1+2\de^2+O(\de^2|b|)+O(\de^4)\big).
\end{aligned}\eeq
Using the expression of $H$ given in \eqref{accor} we compute 
$$\begin{aligned}H(x,b)&=\frac{1}{4\pi}\log \Big(1+|x|^2|b|^2-2b_1x_1-2b_2x_2\Big)\\ &=\frac{1}{4\pi}\log \Big(1+|x-b|^2|b|^2+|b|^4-2|b|^2-2b_1(x_1-b_1)-2b_2(x_2-b_2)+O(|b|^3|x-b|)  \Big) \end{aligned}$$
by which
$$\begin{aligned}&e^{8\pi (H(x, b)-H(b,b))}\\&= \frac{(1+|x|^2|b|^2-2b_1x_1-2b_2x_2)^2}{(1-|b|^2)^4}
\\ &=\frac{\big(1+|x-b|^2|b|^2+|b|^4-2b_1(x_1-b_1)-2b_2(x_2-b_2)-2|b|^2+O(|b|^3|x-b|)\big)^2}{(1-|b|^2)^4}
\\ &=\bigg(1+\frac{|x-b|^2|b|^2-2b_1(x_1-b_1)-2b_2(x_2-b_2)+O(|b|^3|x-b|)}{(1-|b|^2)^2}
\bigg)^2
\\ &= \Big(1-2b_1(x_1-b_1)-2b_2(x_2-b_2)+O(|b|^3|x-b|)+O(|b|^2|x-b|^2) 
\Big)^2
\\ &=1-4b_1(x_1-b_1)-4b_2(x_2-b_2)+O(|b|^2|x-b|^2)+O(|b|^3|x-b|).
\end{aligned}$$
Then \eqref{proprop} becomes
\beq\label{proprop1}\begin{aligned}\frac{\la}{4} V(x^\frac12) e^{PW_{\la}}&=(1+2\de^2)\frac{V(x^\frac12)}{ V (b^\frac12)}e^{W_{\la}}-4\frac{V(x^\frac12)}{ V (b^\frac12)}e^{W_{\la}}(b_1(x_1-b_1)+b_2(x_2-b_2)\\ &\;\;\;\; +e^{W_{\la}}\Big(O(|b|^2|x-b|^2)+O(|b|^3|x-b|)+O(\de^2|b|)+O(\de^4)\Big).
\end{aligned}\eeq
Using the expansion provided by Lemma \ref{crucru} into \eqref{proprop1},  and the crucial assumption \eqref{techni}, we get the estimate \eqref{stimarafff}.
Observe that \eqref{stimarafff} can be written in more approximate way as $$R_\la=e^{W_\la}\Big(O(\de^2)+O(|b|^2|x-b|)+O(|x-b|^2) +O(|b|^{\frac72}\Big).$$ So, by applying Lemma \ref{aux000} we obtain the $L^p$ estimate. 
\end{proof}

\begin{rmk}\label{banale} We observe that for general coefficients $A_0, A_1, A_2$,  after substituting the expansion of Lemma \ref{crucru} into \eqref{proprop1}  we obtain that the following term  $$e^{W_\la}(2A_0-4)b_1(x_1-b_1)+e^{W_\la}\Big(A_0+\frac{A_2}{2}-4\Big)b_2(x_2-b_2)+e^{W_\la} \frac{A_1}{2}\Big(b_1(x_2-b_2)+b_2(x_1-b_1)\Big)$$ does not vanish  
and actually shall represent the leading term in the estimate of the error $R_\la$. 
This will explain later in Remark \ref{banale1} why 
the result of Theorem \ref{th2} fails in general without the assumption \eqref{techni}.
\end{rmk}

\section{The nonlinear problem: a contraction argument}
In order to solve \eqref{proreg1}, let us consider the following intermediate problem:

\beq\label{inter}\left\{\begin{aligned}&-\Delta(PW_{\la}+\phi)-\frac{\la}{4} V(x^{\frac12}) e^{PW_{\la}+\phi}=\sum_{j=1,2}c_j Z_\la^je^{W_{\la}},\\ &\phi \in H^1_{0}( B_1 ),\;\;\;\; \into \nabla \phi\nabla PZ_\la^jdx=0,\;\; j=1,2.\end{aligned}\right.\eeq

 Then it is convenient to solve as a first step the problem for $\phi$ as a function of $b$.
 
Let us rewrite problem \eqref{inter} in a more convenient way.  In what follows we denote by $N:H^1_{0}( B_1 )\to K^\perp$ the nonlinear operator
$$\begin{aligned}N(\phi)&=\Pi^\perp\({i^*_{p}}\bigg(       \frac{\la}{4} V\big(x^{\frac12}\big) e^{PW_{\la}}(e^{\phi}-1-\phi) \bigg)\).\end{aligned}$$
 Therefore problem \eqref{inter} turns out to be equivalent to the problem

 \beq\label{interop}  L(\phi)+N(\phi)=\tilde R,\quad \phi\in K^\perp\eeq
 where, recalling Lemma \ref{aux00},  $$\tilde R=\Pi^\perp\({i^*_{p}}\big(R_\la\big)\)=
 \Pi^\perp\(PW_{\la} -{i^*_{p}}\bigg(\frac{\la}{4} V\big(x^{\frac12}\big) e^{P W_{\la}}\bigg)\).$$

 We need the following auxiliary lemma.

 \begin{lemma}\label{auxnonl}
 For any $p> 1$ and
 any $\phi_1,\phi_2\in H_{0}^1( B_1 )$ with $\|\phi\|_1,\,\|\phi_2\|<1$ the following holds

\beq\label{skate1}\|e^{\phi_1}-\phi_1-e^{\phi_2}+\phi_2\|_p\leq C(\|\phi_1\|+\|\phi_2\|)\|\phi_1-\phi_2\|,\eeq
 \beq\label{skate2}\|N(\phi_1)-N(\phi_2)\|\leq C\de^{-2\frac{p^2-1}{ p^2}}(\|\phi_1\|+\|\phi_2\|)\|\phi_1-\phi_2\|\eeq
 uniformly for $b$ in a small neighborhood of $0$.
 \end{lemma}
 \begin{proof} 
 
A straightforward computation gives that  the inequality $|e^a-a-e^b+b|\leq e^{|a|+|b|}(|a|+|b|)|a-b|$ holds  for all $a,b\in \R$. Then, by applying H\"older's inequality with $\frac1q+\frac1r+\frac1t=1$, we derive
$$\|e^{\phi_1}-\phi_1-e^{\phi_2}+\phi_2\|_p\leq C\|e^{|\phi_1|+|\phi_2|}\|_{pq}(\|\phi_1\|_{pr}+\|\phi_2\|_{pr})\|\phi_1-\phi_2\|_{pt}$$
 and \eqref{skate1} follows by using Lemma \ref{tmt} and the continuity of the embeddings $H^1_{0}( B_1 )\subset L^{pr}( B_1 )$ and $H^1_0( B_1 )\subset L^{pt}( B_1 )$.
  Let us prove \eqref{skate2}. According to \eqref{isp}  we get
 $$\|{N}(\phi_1)-{N}(\phi_2)\|\leq C\|\la V(x^{\frac12})  e^{PW_{\la}}(e^{\phi_1}-\phi_1-e^{\phi_2}+\phi_2)\|_p,$$
and by H\"older's inequality with $\frac1p+\frac1q=1$,  we derive
$$\begin{aligned}\|{N}(\phi_1)-{N}(\phi_2)\|&\leq C\|\la V(x^{\frac12}) e^{PW_{\la}}\|_{p^2}\|e^{\phi_1}-\phi_1-e^{\phi_2}+\phi_2|\|_{pq}\\ &\leq C\|\la V(x^{\frac12}) e^{PW_{\la}}\|_{p^2}(\|\phi_1\|+\|\phi_2\|)\|\phi_1-\phi_2\|\end{aligned}
$$ by \eqref{skate1}, 
and  the conclusion follows by Lemma \ref{aux00}.
 \end{proof}

Problem \eqref{inter} or, equivalently, problem \eqref{interop} turns out to be solvable for any choice of point $b $ in a small neighbourhood of $0$,  
provided
that $\la$ is sufficiently small. Indeed we have the following result.

\begin{prop}\label{nonl} Assume ${\mathrm{(H1)-(H2)-(H3)}}$ and \eqref{techni}  hold and let $\e>0$ be a fixed small number. Then there exists $\la_0>0$ such that for any $\la\in (0,\la_0)$ and any $b\in\R^2$ with $|b|\leq \de^\frac23$ 
   there is a unique $\phi_\la=\phi_{\la,b}\in K^\perp$ satisfying \eqref{inter} for some $c_1,c_2\in \R$ and
\beq\label{nonll}\|\phi_{\la}\|\leq  \de^{2-\e}
.\eeq

Moreover the map $b\mapsto \phi_{\la, b}\in H^1_0(B_1)$ is continuous. 
\end{prop}
\begin{proof} 
 Since problem \eqref{interop} is equivalent to problem \eqref{inter},
we will show that problem \eqref{interop} can be solved via a contraction mapping argument. Indeed, in virtue of Proposition \ref{ex}, let us introduce the map
$$T:=L^{-1}(\tilde R-N(\phi)),\quad \phi\in K^\perp.$$
 Let us fix 
$p>1$ sufficiently close to 1. 
 By \eqref{isp} and Proposition \ref{aux000}, if $|b|\leq\de^\frac23$ we get 
\beq\label{non1}\|\tilde R\|\leq  C\de^{2-\frac\e2} 
.\eeq
Next, by \eqref{skate2},
\beq\label{non2}\|N(\phi_1)-N(\phi_2)\|\leq C\de^{-\frac{\e}{2}}(\|\phi_1\|+\|\phi_2\|)\|\phi_1-\phi_2\|\quad \forall \phi_1,\phi_2\in H_{0}^1( B_1 ), \|\phi_1\|,\|\phi_2\|<1.\eeq In particular, by taking $\phi_2=0$,
\beq\label{non3}\|N(\phi)\|\leq C\de^{-\frac{\e}{2}}\|\phi\|^2\quad \forall \phi\in H_{0}^1( B_1 ), \|\phi\|<1.\eeq

We claim that $T$ is a contraction map over the ball $${\cal B}:=\Big\{\phi\in K^\perp\,\Big|\, \|\phi\|\leq \de^{2-\e} 
\Big\}$$ provided that $\la$ is small enough. Indeed, combining  Proposition \ref{ex}, \eqref{non1}, \eqref{non2}, \eqref{non3}, for any $\phi\in {\cal B}$  we have
$$\begin{aligned}\|T(\phi)\|&\leq C|\log\de| \big(\|\tilde R\|+\|N(\phi)\|\big)\leq  C|\log\de|\de^{2-\frac\e2} <\de^{2-\e}
.\end{aligned}$$  Similarly, for any $\phi_1,\phi_2\in{\cal B}$  
$$\|T(\phi_1)-T(\phi_2)\|\leq C|\log\de|\|N(\phi_1)-N(\phi_2)\|\leq C\de^{-\frac{\e}{2}}|\log\de| (\|\phi_1\|+\|\phi_2\|)\|\phi_1-\phi_2\|\leq \frac12\|\phi_1-\phi_2\|.$$
Uniqueness of solutions implies continuous dependence of $\phi_\la=\phi_{\la,b}$  on $b$.
\end{proof}

\section{Proof of Theorems \ref{th2} and  Theorem \ref{main2}}
During this section we assume that the crucial assumption ${\mathrm{(H1)-(H2)-(A3)}}$  and \eqref{techni} of Theorem \ref{th2} hold. 

After problem \eqref{inter} has been solved according to Proposition \ref{nonl}, then we find a solution to the original problem \eqref{proreg1} if $b\in \R^2$   is such that $|b|\leq \de^\frac23$ and  $$c_1=c_2=0.$$ 
 Let us find the condition satisfied by $b$ in order to get  $c_1, c_2$  equal to zero.

\subsection*{Proof of Theorem \ref{main2}} 
We multiply the equation in \eqref{inter} by   $PZ_\la^i$ and integrate over $B_1$:
\beq\label{masca}\begin{aligned}\into \nabla (PW_\la+\phi_{\la}) \nabla PZ^i_{\la} dx &-\frac{\la}{4} \into V\big(x^\frac12\big) e^{P W_\la+\phi_{\la}}PZ^i_{\la} dx\\&=\sum_{h=1,2}c_h \into Z_\la^h e^{W_\la}PZ_\la^i dx.\end{aligned}
\eeq
The object  is now to expand each integral of the above identity and analyze the  leading term. In the remaining part of the section all the estimates  hold uniformly for $|b|\leq \de^\frac23$, without further notice.  

 Let us begin by observing that  the orthogonality in \eqref{inter} gives
\beq\label{masca1} \into \nabla \phi_{\la} \nabla PZ^i_{\la} dx =\into e^{W_\la} \phi_{\la}Z_\la^i dx=0\eeq
and, by \eqref{pzi}, \beq\label{masca2}\begin{aligned}\into Z_\la^h e^{W_\la}PZ_\la^i dx
= \intr \frac{8z_iz_h}{(1+|z|^2)^4} dz+o(1) =\left\{\begin{aligned}&\frac23\pi+o(1) &\hbox{ if }& h=i\\ & o(1)& \hbox{ if }& h\neq i\end{aligned}\right.
\end{aligned}\eeq where we have used that $\intr \frac{z_i^2}{(1+|z|^2)^4} dz= \frac23 \pi$ and $\intr \frac{z_1z_2}{(1+|z|^2)^4} dz= 0.$
Using the definition of $R_\la$ in Lemma \ref{aux000}, \eqref{masca1} and \eqref{masca2}, then \eqref{masca}  becomes 

\beq\label{mascara}\begin{aligned}\into R_\la PZ^i_{\la} dx &+\frac{\la}{4} \into V(x^{\frac12}) e^{P W_\la}(e^{\phi_\la}-1)PZ^i_{\la} dx=\left\{\begin{aligned}&-\frac23\pi+o(1) &\hbox{ if }& h=i\\ & o(1)& \hbox{ if }& h\neq i\end{aligned}\right.
.\end{aligned}
\eeq
Let us first estimate the term containing the function $\phi_\la$: recalling \eqref{masca1} 
\beq\label{mos}\begin{aligned}\frac{\la}{4} \into V(x^{\frac12}) e^{P W_\la}(e^{\phi_\la}-1)PZ^i_{\la} dx&= \into R_\la
(e^{\phi_\la}-1)PZ^i_{\la} dx
\\ &\;\;\;\;+ \into  e^{W_\la}(e^{\phi_\la}-1-\phi_\la)PZ^i_{\la}  dx\\ &\;\;\;\;+ \into  e^{W_\la}\phi_\la(PZ^i_{\la}-Z_\la^i)  dx.
\end{aligned}\eeq
Now, let us fix $\e>0$ sufficiently small and $p>1$ sufficiently close to 1.  
Next let  $1<q<\infty$ be such that  $\frac1p+\frac1q=1$. Then, 
\eqref{skate1} with $\phi_2=0$ and Proposition \ref{nonl} give
$$\|e^{\phi_\la}-1-\phi_\la\|_q\leq  C\|\phi_\la\|^2\leq C \de^{4-2 \e} 
$$
and, consequently, 
\beq\label{tmtsur}\|e^{\phi_\la}-1\|_q\leq C\|\phi_\la\|\leq C\de^{2-\e}
.\eeq
Therefore,  Lemma \ref{aux00}  implies
 \beq\label{capra1}\begin{aligned} 
 \into e^{W_\la}(e^{\phi_\la}-1-\phi_\la)PZ^i_{\la} dx&
= O(\|  e^{ W_\la}(e^{\phi_\la}-1-\phi_\la)\|_1)=
O(\|   e^{W_\la} \|_p\|e^{\phi_\la}-1-\phi_\la\|_q)\\ &=O\Big(\de^{4-2\frac{p-1}{p}-2\e} \Big).
\end{aligned}\eeq
Now, by Lemma \ref{aux000}
 \beq\label{capra2}\begin{aligned}  \into R_\la (e^{\phi_\la}-1) 
 PZ^i_{\la} dx&=
O\big(\big\|  
R_\la (e^{\phi_\la}-1)\big\|_1\big)
=O\big(\big\| 
R_\la\|_p\|e^{\phi_\la}-1\|_q\big)\\ &= O\Big(\de^{4-2\frac{p-1}{p}-\e}\Big).\end{aligned}\eeq
Finally by Lemma \ref{pzirafff}  and  Lemma \ref{aux00}, using that $|b|\leq \de^{\frac23}$,
 \beq\label{capra3}\begin{aligned}  \into  e^{W_\la}\phi_\la(PZ^i_{\la}-Z_\la^i)  dx&=-\de \into e^{W_\la}\phi_\la(x_1-b_1)  dx+O\bigg(\de^\frac53\into e^{W_\la}|\phi_\la| dx\bigg)\\ &=O(\de\||x-b|e^{W_\la}\|_p\|\phi_\la\|)+O(\de^{\frac53}\|e^{W_\la}\|_p\|\phi_\la\|)
 \\ &=O(\de^{4-\e-2\frac{p-1}{p}})+O(\de^{\frac{11}{3}-2\frac{p-1}{p}-\e})
. \end{aligned}\eeq
 By inserting \eqref{capra1}-\eqref{capra2}-\eqref{capra3} into \eqref{mos}, 
 we obtain
 \beq\label{capra}\la \into V\big(x^\frac1\al\big) e^{P W_\la}(e^{\phi_\la}-1)PZ^i_{\la} dx= O(\de^3) \eeq provided that that $\e$ is chosen sufficiently close to $0$ and $p$ sufficiently close to $1$.
Next, by \eqref{stimarafff}, using  Lemma \ref{robin2} and Lemma \ref{robin22}, 
we get
 $$\begin{aligned}\into R_\la PZ_\la^1dx &= 2\pi\de\bigg(
 3b_1^2D_0+D_1b_1b_2+\frac{15D_0-D_2}{4}\de^2\log\frac1\de + \frac{3D_0-D_2}{4}b_2^2\bigg)\\ &\;\;\;\;+O(\de^3)+O(\de^2| b|)+O(|b|^\frac72)+O(|b|^3\de).
\end{aligned}$$
 $$\begin{aligned}\into R_\la PZ_\la^2dx &= 2\pi\de\bigg(
 \frac{D_1}{2}b_1^2+\frac{3D_0-D_2}{2}b_1b_2+\frac{10D_1+3D_3}{8}\de^2\log\frac1\de +3 \frac{2D_1+D_3}{8}b_2^2
 \bigg)\\ &\;\;\;\;+O(\de^3)+O(\de^2| b|)+O(|b|^{\frac72})+O(|b|^3\de).
\end{aligned}$$
By inserting  the above identity and \eqref{capra} into \eqref{mascara}  we deduce 
 \beq\label{strep1}2\pi\de^2\log\frac1\de  F\bigg(\frac{ b}{\de\sqrt{\log\frac1\de}}\bigg)+O(\de^3)+O(\de^2 |b|)+O(|b|^{\frac72})+O(|b|^3\de)= -\frac23\pi c+o(|c|), \eeq
 where $c=(c_1,c_2)$ and $F:\R^2\to\R^2$ denotes the vector field defined in \eqref{vector}.

Now let $\xi_0\neq 0$ be a zero for $F$ which is stable under uniform perturbations according to Theorem \ref{th2}; then \eqref{strep1} gives that the following holds
 \beq\label{strep2}\begin{aligned} \de^2\log\frac1\de  F\bigg(\frac{ b}{\de\sqrt{\log\frac1\de}}\bigg)+o\bigg(\de^2\log\frac1\de \bigg)
 = -\frac{c}{3} +o(|c|)\; \hbox{ unif. for }|b|\leq 2|\xi_0|\de\sqrt{\log\frac1\de}.
 \end{aligned}\eeq
 Now, setting  $$\tilde b=\frac{b}{\de\sqrt{\log\frac1\de}},$$ we rewrite \eqref{strep2} as 
\beq\label{strep3}\begin{aligned} \de^2\log\frac1\de  \bigg(F(\tilde b)+o(1)\bigg) 
= -\frac{c}{3}+o(|c|)\; \hbox{ unif. for }|\tilde b|\leq 2|\xi_0|.
\end{aligned}\eeq
  The continuity of the map $b\mapsto \phi_\la=\phi_{\la,b}$ guaranteed by Proposition \ref{nonl} implies that the left hand side of \eqref{strep3} is continuous too. So, the uniform stability gives that, if $\eta>0$ is sufficiently small, then for $\la$ small enough the left hand side of \eqref{strep3} has a zero $\tilde b_\la$ with $ |\tilde b_\la-\xi_0|\leq \eta$ or, equivalently, the left hand side of \eqref{strep2} has a zero $b_\la$ with $\big|\frac{b_\la}{\de\sqrt{\log\frac1\de}}-\xi_0\big|\leq \eta.$ The arbitrariness of $\eta$ implies
 $$b_\la= \xi_0 \de\sqrt{\log\frac1\de}(1+o(1)).$$

 \begin{rmk}\label{banale1} We point put that, for general coefficients $A_0,A_1, A_2$, then the error term $R_\la$  reduces to the expression in Remark  \ref{banale} at the leading part and,   thanks to Lemma \ref{robin2}, when we multiply it against $PZ_i^\la$ we actually obtain
 $$2\pi\de (2A_0-4)b_1+\pi A_1\de b_2+h.o.t, \quad 2\pi\de\Big(A_0+\frac{A_2}{2}-4\Big)b_2+\pi A_1\de b_1+h.o.t.$$
 which in general 
admits the only  trivial zero $ b = 0$  for the leading term, 
 so we are unable to catch a non-simple blow-up solution without the assumption \eqref{techni}. 
 \end{rmk}
 
 \subsection{Proof of Theorems  \ref{th2}.} Theorem \ref{main2} provides a solution to the problem \eqref{proreg1} of the form $$w_\la=PW_\la+\phi_\la$$ where $\phi_\la=\phi_{\la, b_\la}\in H^1_0(B_1)$ satisfies \eqref{nonll} and  $b=b_\la$ satisfies \eqref{saltria}. 

Moreover, using  \eqref{tmtsur} and Lemma \ref{aux00}, by H\"older's inequality with $\frac1p+\frac1q=1$ we get
$$\begin{aligned}\la\| V(y^{\frac12})(e^{{w}_\la}- e^{{PW}_\la})\|_1&=\la\|V(y^\frac12)e^{{PW}_\la}(e^{{\phi}_\la}-1)\|_{1}
\\ &\le\la \|e^{{PW}_\la}\|_p\|e^{\phi_\la}-1\|_q\\ &=O(\de^{2-2 \frac{p-1}{ p}-\e })=o(1), \end{aligned}$$ if $p$ is chosen sufficiently close to 1 and $\e$ sufficiently close to $0$. Similarly, by Proposition \ref{aux000},
$$\begin{aligned}\Big\|\frac{\la}{4} V(y^{\frac12})e^{{PW}_\la}-e^{W_\la}\Big\|_1&=\|R_\la\|_1
=
O(\de^{2-2 \frac{p-1}{ p}})=o(1).\end{aligned}$$
Therefore 
$$\begin{aligned}\Big\|\frac{\la }{4}V(y^{\frac12})e^{w_\la}-e^{W_\la}\Big\|_1=o(1).\end{aligned}$$
Clearly, by \eqref{chva} and \eqref{chva1},  $$u_\la(x)=w_\la(x^2)-4\pi  G(x,0)= w_\la(x^2)-2\log\frac{1}{|x|} $$ solves equation \eqref{eq0}
and 
$$\begin{aligned}\|\la V(x)e^{u_\la(x)}-4 |x|^{2}e^{W_\la(x^2)}\|_1&=4\Big\|\frac{\la}{4} |x|^{2} V(x)e^{{w}_\la(x^2)}-|x|^{2}e^{W_\la(x^2)}\Big\|_1
\\ &= 2\Big\|\frac{\la}{4} V(y^\frac12)e^{{w}_\la(y)}-e^{W_\la(y)}\Big\|_1=o(1)\end{aligned}$$ by Lemma \ref{copy}. Hence,
recalling \eqref{quantum} and Lemma \ref{copy},
$$\begin{aligned}\la \into V(x)e^{u_\la}dx&=4
\intr|x|^{2}V (x)e^{{W}_\la(x^2)}dx+o(1)\\ &=2 \intr V(y^{\frac12})e^{W_\la(y)} dy+o(1)=
16\pi+o(1). \end{aligned}$$
Similarly for every neighborhood $U$ of $0$ $$\la \int_U V(x) e^{u_\la}dx\to 16\pi. $$Theorem \ref{th2} is thus completely proved by setting $\mu^2=\de$.

\appendix
\renewcommand{\theequation}{\Alph{section}.\arabic{equation}}

\section{}
In this appendix we derive some crucial integral estimates which arise  in the asymptotic expansion of the energy of approximate solution $PW_{\la}$. 

\begin{lemma}\label{aux00app} The following holds:
$$\into e^{W_{\la}}  |x-b|dx=O(\de),\quad\into e^{W_{\la}}  |x-b|^2dx=16\pi\de^2|\log\de|+O(\de^2),\quad \into e^{W_{\la}} |x-b|^3 dx=O(\de^2)
$$
uniformly for $b$ in a small neighborhood of $0$. 
\end{lemma}

\begin{proof} We compute $$\begin{aligned}\into e^{W_{\la}}  |x-b|dx &
  \leq 8\de\intr \frac{1}{(1+|z-\de^{-1} b|^2)^2}|z-\de^{-1}b|dz
 =8\de\intr \frac{|z|}{(1+|z|^2)^2}dz
\end{aligned}$$
and the first estimate follows. In order to show the second estimate let us observe that $B(b, 1-|b|)\subset B(0,1)\subset B(b, 1+|b|)$, so we compute
$$\begin{aligned}\into e^{W_{\la}}|x-b|^2 dx &= 8\int_{B_1 }\frac{\de^2|x-b|^2}{(\de^2+|x-b|^2)^2} dx \\ &=
8 \int_{B(b, 1-|b|)} \frac{\de^2|x-b|^2}{(\de^2+|x-b|^2)^2}dx +O\bigg(\int_{ B(b, 1+|b|)\setminus B(b,1-|b|)}\frac{\de^2|x-b|^2}{(\de^2+|x-b|^2)^2}dx\bigg)
\\ &= 8 \int_{B(0, 1-|b|)} \frac{\de^2|x|^2}{(\de^2+|x|^2)^3}dx +O\bigg(\int_{ B(0, 1+|b|)\setminus B(0,1-|b|)}\frac{\de^2|x|^2}{(\de^2+|x|^2)^2}dx\bigg)\\ &= 
8\de^2\int_{|z|\leq \frac{1-|b|}{\de}} \frac{|z|^2}{(1+|z|^2)^2}dz+O\bigg(\de^2\int_{ \frac{1-|b|}{\de}\leq |z|\leq \frac{1+|b|}{\de}}\frac{1}{|z|^2}dz\bigg) \\ &=
8\de^2\int_{|z|\leq \frac{1-|b|}{\de}} \frac{1}{1+|z|^2}dz+O(\de^2)\\ &=
16\pi\de^2|\log\de|+O(\de^2).\end{aligned}$$ 
In order to prove the third estimate,
let $R>1$ so that $ B(0,1) \subset B(b, R)$ if $b$ lies in a small neighborhood of $0$. Then,
$$\begin{aligned}\into e^{W_{\la}}  |x-b|^3dx &  =8\de^{3}\int_{|z|\leq \frac{1 }{\de}} \frac{1}{(1+|z-\de^{-1} b|^2)^2}|z-\de^{-1}b|^3dz\\ &
 \leq 8\de^{3}\int_{B(0,\frac{R}{\de})} \frac{|z|^3}{(1+|z|^2)^2}dz\leq C\de^{2}.
\end{aligned}$$

\end{proof}

Since the key part in the proof of Theorem \ref{main2} relies in testing the equation \eqref{inter} with $PZ_\la^i$ in order to catch the leading terms, a crucial step consists in the evaluation of some integral estimates, as  provided by the following lemma.
\begin{lemma}\label{robin2} The following holds for $i,j=1,2$:
$$\into e^{W_{\la}} PZ^i_\la dx =O(\de),$$ $$\into e^{W_{\la}} PZ^i_\la( x_i-b_i)dx=2\pi\de+O(\de^{2}),\quad\into e^{W_{\la}} PZ^i_\la( x_j-b_j)dx=O(\de^{2})\quad i\neq j,
$$
$$\into e^{W_{\la}} |PZ^i_\la|| x-b|^2dx=O(\de^{2}),
$$
$$\into e^{W_{\la}} PZ^i_\la(x_i-b_i)^3dx=6\pi\de^3\log\frac{1}{\de}+O(\de^3)\quad \into e^{W_{\la}} PZ^i_\la(x_j-b_j)^3dx=O(\de^3)\quad i\neq j,$$ $$\quad \into e^{W_{\la}} PZ^i_\la(x_j-b_j)^2( x_i-b_i)dx=2\pi\de^3\log\frac{1}{\de}+O(\de^3) \quad i\neq j,$$
$$\into e^{W_{\la}} PZ^i_\la(x_i-b_i)^2( x_j-b_j)dx=O(\de^3) \quad i\neq j,$$
 $$\into e^{W_{\la}} |PZ^i_\la|| x-b|^{\frac72}dx=O(\de^{3})$$ 
uniformly for $b$ in a small neighborhood of $0$. 
\end{lemma}
\begin{proof} 
We compute $$\begin{aligned}\into e^{W_{\la}} Z^i_\la dx& =8\int_{|z|\leq \frac{1 }{\de}}\frac{1}{(1+|z-\de^{-1}b|^2)^3}(z_i-\de^{-1}b_i) dz\\ &=8\intr\frac{1}{(1+|z-\de^{-1}b|^2)^3}(z_i-\de^{-1}b_i) dz+O(\de^{3})
\\ &=8\intr\frac{z_i}{(1+|z|^2)^3} dz+O(\de^{3})
=O(\de^3),
\end{aligned} $$
since $\intr\frac{z_i}{(1+|z||^2)^3} dz=0$  by oddness. Next
$$\begin{aligned}\into e^{W_{\la}} Z^i_\la(x_i-b_i)dx&=8\de\int_{|z|\leq \frac{1 }{\de}}\frac{1}{(1+|z-\de^{-1}b|^2)^3}(z_i-\de^{-1}b_i)^2 dz \\ &=8\de\intr\frac{1}{(1+|z-\de^{-1}b|^2)^3}(z_i-\de^{-1}b_i)^2 dz+O(\de^{3})
\\ &= 8 \de\intr\frac{z_i^2}{(1+|z|^2)^3}dz+O(\de^{3})\\ &=2\pi \de +O(\de^{3})\end{aligned} $$ where we have used  the identity $\intr \frac{(z_i)^2}{(1+|z|^2)^3} =\frac{1}{2} \intr \frac{|z|^2}{(1+|z|^2)^3} =\frac{\pi}{4} $. Similarly for $i\neq j$ $$\begin{aligned}\into e^{W_{\la}} Z^i_\la(x_j-b_j)dx &= 8 \de\intr\frac{z_iz_j}{(1+|z|^2)^3}dz+O(\de^{3})=O(\de^{3})\end{aligned} $$ since $\intr\frac{z_iz_j}{(1+|z|^2)^3}dz=0. $ Next, $$\begin{aligned}\into e^{W_{\la}}  |Z^i_\la|| x-b|^2dx &
  \leq 8\de^2\intr \frac{|x-b|^3}{(\de^2+|x- b|^2)^3}dx
 =8\de^2\intr \frac{|z|^3}{(1+|z|^2)^3}dz\leq C\de^2.
\end{aligned}$$

Using that $B(b, 1-|b|)\subset B(0,1)\subset B(b, 1+|b|)$,  we compute
$$\begin{aligned}&\into e^{W_{\la}} Z^i_\la(x_i-b_i)^3dx\\ &= 8\de^3\int_{B_1 }\frac{(x_i-b_i)^4 }{(\de^2+|x-b|^2)^3}dx 
\\ &=
8\de^3\int_{B(b, 1-|b|)} \frac{(x_i-b_i)^4}{(\de^2+|x-b|^2)^3}dx +O\bigg(\de^3\int_{ B(b, 1+|b|)\setminus B(b,1-|b|)}\frac{|x-b|^4}{(\de^2+|x-b|^2)^3}dx\bigg)
\\ &= 8\de^3 \int_{B(0, 1-|b|)} \frac{( x_i)^4}{(\de^2+|x|^2)^3}dx +O\bigg(\de^3\int_{ B(0, 1+|b|)\setminus B(0,1-|b|)}\frac{|x|^4}{(\de^2+|x|^2)^3}dx\bigg)\\ &= 
8\de^3\int_{|z|\leq \frac{1-|b|}{\de}} \frac{(z_i)^4}{(1+|z|^2)^3}dz+O(\de^3)
\\ &=6\pi \de^3|\log\de|+O(\de^3)\end{aligned}$$
where we have used the identity $\int_{|z|\leq r} \frac{(z_i)^4}{(1+|z|^2)^3}dz= \frac38\pi\log(1+r^2)+\frac34\frac{\pi}{1+r^2}-\frac{3}{16}\frac{\pi}{(1+r^2)^2}-\frac{9}{16}\pi$.

Similarly, for $i\neq j$
$$\begin{aligned}\into e^{W_{\la}} Z^i_\la(x_j-b_j)^3dx &= 
8\de^3\int_{|z|\leq \frac{1-|b|}{\de}} \frac{z_i(z_j)^3}{(1+|z|^2)^3}dz+O(\de^3) 
 =O(\de^3)\end{aligned}$$
since  $\int_{|z|\leq r} \frac{z_i(z_j)^3}{(1+|z|^2)^3}dz=0$.

Next, for $i\neq j$ 
$$\begin{aligned}\into e^{W_{\la}} Z^i_\la(x_j-b_j)^2( x_i-b_i)dx &=8\de^3\int_{|z|\leq \frac{1-|b|}{\de}} \frac{(z_i)^2(z_j)^2}{(1+|z|^2)^3}dz+O(\de^3)\\ &=2\pi \de^3|\log\de|+O(\de^3)
\end{aligned}$$where the last equality follows by  $\int_{|z|\leq r} \frac{(z_i)^2(z_j)^2}{(1+|z|^2)^3}dz= \frac{\pi}{8}\log(1+r^2)+\frac14\frac{\pi}{1+r^2}-\frac{1}{16}\frac{\pi}{(1+r^2)^2}-\frac{3}{16}\pi$.
Similarly for $i\neq j$ $$\begin{aligned}\into e^{W_{\la}} Z^i_\la(x_i-b_i)^2( x_j-b_j)dx &=8\de^3\int_{|z|\leq \frac{1-b}{\de}} \frac{(z_i)^3z_j}{(1+|z|^2)^3}dz+O(\de^3) =O(\de^3)
\end{aligned}$$by  $\int_{|z|\leq r} \frac{(z_i)^3z_j}{(1+|z|^2)^3}dz=0$.
Finally
$$\begin{aligned}\into e^{W_{\la}}  |Z^i_\la|| x-b|^{\frac72}dx &
  \leq 8\de^2\int_{B(b, 1+|b|)} \frac{\de|x-b|^{\frac92}}{(\de^2+|x- b|^2)^3}dx
 =8\de^{\frac72}\int_{|z|\leq \frac{1+|b|}{\de}} \frac{|z|^{\frac92}}{(1+|z|^2)^3}dz\leq C\de^3.
\end{aligned}$$
Taking into account that $PZ^i_\la=Z^i_\la+O(\de)$ by \eqref{pzi}, and recalling Lemma \ref{aux00app},  the above integral estimates give the thesis. 
\end{proof}

In order to derive next integral estimate we need to expand  the projections $PZ_\la^i$ to a higher order with respect to \eqref{pzi}.

\begin{lemma}\label{pzirafff} For $i=1,2$ the following holds:
$$PZ_\la^i=
Z_\la^i (x)-\de(x_i-b_i)+O(\de^3)+O(\de| b|)\hbox{ in } B_1$$
uniformly for $b$ in a small neighborhood of $0$.
\end{lemma}
\begin{proof}
Let us consider $i=1$. Observe that 
$$\begin{aligned}\hbox{if } |x|=1: \;\;Z_\la^1(x)&=\frac{\de(x_1-b_1)}{\de^2+|x-b|^2}=\frac{\de(x_1-b_1)}{1+\de^2+O(|b|)
}\\ &=\de(x_1-b_1)\Big(1+O(|b|)+O(\de^2) \Big)
\\ &=\de(x_1-b_1)+
O(|b|\de)+O(\de^3) .\end{aligned}$$
Therefore, if we set  $$ \hat{Z}_\la^1:= Z_\la^1 (x)-\de(x_1-b_1),$$ we get $$\hat Z_\la^1(x)=O(\de^3)+O(\de| b|) \hbox{ if }|x|=1$$ and 
$$-\Delta \hat Z_\la^1 (x)=-\Delta Z_\la^1(x)=\Delta PZ_\la^1(x) \hbox{ in }B_1.$$
Hence, since by construction $PZ_\la^1=0$ for $|x|=1$, the maximum principle applies and gives 
$$PZ_\la^1= \hat Z_\la^1 +O(\de^3)+O(\de| b|) =Z_\la^1 (x)-\de(x_1-b_1)+O(\de^3)+O(\de |b|)\hbox{ in } B_1.$$
\end{proof}

\begin{lemma}\label{robin22} Let $P$ be a homogeneous polynomial of degree 2. Then the following holds:$$\into e^{W_{\la}} PZ^i_\la P(x-b)dx=O(\de^3)+O(\de^2 |b|)\quad i=1,2$$
uniformly for $b$ in a small neighborhood of $0$. 

\end{lemma}
\begin{proof} We compute
\beq\label{mes}\begin{aligned}\into e^{W_{\la}} Z^i_\la P(x-b) dx&= 8\de^2\int_{|z|\leq \frac{1 }{\de}}\frac{P(z-\de^{-1}b)}{(1+|z-\de^{-1}b|^2)^3}(z_i-\de^{-1}b_i) dz\\ &= 8\de^2\intr \frac{P(z)}{(1+|z|^2)^3}z_idz+O(\de^3)=O(\de^3)\end{aligned}\eeq where we have used that  $\intr \frac{P(z)}{(1+|z|^2)^3}z_idz=0$    by oddness. 
Taking into account of Lemma \ref{pzirafff} 
we get 
$$\begin{aligned}\into e^{W_{\la}} PZ^i_\la P(x-b) dx&= \into e^{W_{\la}} Z^i_\la P(x-b) dx-\de \into e^{W_{\la}}  (x_i-b_i) P(x-b) dx\\ &\;\;\;\;+ \Big(O(\de^3)+O(\de |b|)\Big) \into e^{W_{\la}} |x-b|^2 dx
\\ &= \into e^{W_{\la}} Z^i_\la P(x-b) dx+O(\de) \into e^{W_{\la}}  |x-b|^3 dx\\ &\;\;\;\;+ \Big(O(\de^3)+O(\de |b|)\Big) \into e^{W_{\la}} |x-b|^2 dx\end{aligned}$$
and the thesis follows by \eqref{mes}, and recalling Lemma \ref{aux00app}. 
\end{proof}

Finally we deduce some integral identities associated to the change of variable $x\mapsto x^\al$ which appears frequently when dealing with $\al$-symmetric functions. \begin{lemma}\label{copy} Let $\al\in\N$, $\al\geq2$, and let $f\in L^1(B_1)$. 
Then we have that $|x|^{2(\al-1)}f(x^\al)\in L^1(B_1) $ and $$\int_{B_1} |x|^{2(\alpha-1)} f(x^\al) dx=\frac1\al\int_{B_1} f(y)dy.$$
\end{lemma}
\begin{proof} It is sufficient to prove the thesis for a smooth function $f$. Using  the polar coordinates $(\rho,\theta)$ and then applying the change of variables $(\rho',\theta')=(\rho^\al,\al\theta)$
$$\begin{aligned} \int_{B_1} |x|^{2(\alpha-1)} f(x^\al) dx&=\int_0^{+\infty} d\rho\int_0^{2\pi} \rho^{2\al-1}f(\rho^\al e^{{\rm i}\al \theta}) d\theta
\\ &= \frac{1}{\al^2}\int_0^{+\infty} d\rho'\int_0^{2\al \pi}  \rho'f(\rho' e^{{\rm i}\theta'}) d\theta' \\ &=\frac{1}{\al}\int_0^{+\infty} d\rho'\int_0^{2\pi}  \rho'|f(\rho' e^{{\rm i}\theta'})|^2 d\theta'
\\& =\frac{1}{\al}\into f(y) dy.\end{aligned}$$

\end{proof}

\end{document}